\documentclass[a4paper]{amsart}
\usepackage[latin1]{inputenc}
\usepackage[english]{babel}
\usepackage{amssymb}
\usepackage{amsmath}
\usepackage{amsthm}
\usepackage{latexsym}
\usepackage[breaklinks=true,  colorlinks=true,  linkcolor=blue,  citecolor=blue,  urlcolor=blue]{hyperref}
\usepackage{graphicx}
\usepackage{pdfpages}
\newtheorem{definition}{Definition}[section]
\newtheorem{theorem}{Theorem}
\newtheorem{proposition}[definition]{Proposition}

\newtheorem{lemma}[definition]{Lemma}
\newtheorem{corollary}[definition]{Corollary}
\theoremstyle{definition}
\newtheorem{rem}[definition]{Remark}
\renewcommand{\phi}{\varphi}
\DeclareMathOperator{\vol}{vol}
\DeclareMathOperator{\conv}{conv}
\DeclareMathOperator{\dist}{dist}
\DeclareMathOperator{\Id}{Id}
\DeclareMathOperator{\spann}{span}
\newcommand{\skalar}[2]{\left\langle #1, #2 \right\rangle}
\newcommand{\norm}[1]{\left\Vert #1 \right\Vert}
\newcommand{\D}{\mathrm{d}}
\newcommand{\R}{\mathbb R}
\begin{document}
\title[Sections of Simplices]{Sections of the regular simplex - Volume formulas and estimates}
\author{Hauke Dirksen}
\address{Department of Mathematics, Kiel University}
\email{hauke.dirksen(at)gmx.de}
\date{September 21, 2015 (revised March 1, 2016)}
\begin{abstract}
We state a general formula to compute the volume of the intersection of the regular $n$-simplex with some $k$-dimensional subspace. 
It is known that for central hyperplanes the one through the centroid containing $n-1$ vertices gives the maximal volume. We show that, for fixed small distances of a hyperplane to the centroid,  the hyperplane containing $n-1$ vertices is still volume maximizing. The proof also yields a new and short argument for the result on central sections. With the same technique  we give a partial result for the minimal central hyperplane section.
Finally, we obtain a bound for $k$-dimensional sections.
\end{abstract}


\maketitle
\section{Introduction}
Given a convex body $K \subset \R^{n} $ and some subspace $H$, how to compute the volume of the intersection $H\cap K$? How to find the maximal or minimal sections?
These questions have been considered for various convex bodies. A first example is the unit cube intersected with central hyperplanes. Explicit formulas were already found by Laplace. The question of the minimal and the maximal section were answered by D. Hensley \cite{Hensley1979} resp. K. Ball \cite{Ball1986}. Since then several other bodies and modified questions have been considered. We mention a few examples: $\ell_p$-balls \cite{Meyer1988}, \cite{Koldobsky2005}, complex cubes \cite{Oleszkiewicz2000}; also general $k$-dimensional sections \cite{Ball1989} and non-central sections \cite{Moody2013} as well as taking other than Lebesgue measures \cite{Koenig2013d} have been investigated.

Here we are interested in the regular simplex. S. Webb \cite{Webb1996} gave a formula for central hyperplane sections. He also proved that the maximal central section is the one containing $n-1$ vertices and the centroid. The question of the minimal central hyperplane section is not completely solved yet. P. Filliman stated that his methods can be used to prove that the section parallel to a face is minimal \cite{Filliman1992}. But he gave no precise arguments. P. Brzezinski proved a lower bound which differs from the conjectured minimal volume by a factor of approximately $1.27$ \cite{Brzezinski2013}.

In this paper we consider  $k$-dimensional and also non-central sections. We extend S. Webb's formula and show that, for hyperplanes for fixed small distances from the centroid to the hyperplane, the one containing $n-1$ vertices still gives maximal volume. Concerning the minimal section we give a result that supports the conjecture. For dimensions up to $4$ we prove the conjecture. Using Brascamp-Lieb inequality we also give a bound on the volume of $k$-dimensional sections. In the final chapter we consider irregular simplices and construct an interesting example, using an idea by Walkup \cite{Walkup1968}. It has the property that all its faces have smaller volume than some central section. This is not true for the regular simplex and for any simplex in dimensions up to $4$.

We use the following notations. Let \[S:=\bigg\{ x=(x_1,\dots, x_{n+1}) \in \R^{n+1} \mid \sum_{j=1}^{n+1} x_j =1, \ x_j\geq 0\bigg\}.\] $S$ is the embedded regular $n$-Simplex with $n+1$ vertices and side length $\sqrt{2}$. The Euclidean norm is denoted by $\norm{x}$, the standard scalar product by $\skalar{x}{y}$. The distance of two sets is given by $\dist (A,B):=\inf \{\norm{a-b} \mid a\in A, b\in B\}$, especially $\dist (x,A):=\dist (\{x\},A)$.
 If $H$ is a $k$-dimensional (affine) subspace and  $A\subset H$, the $k$-volume of $A$ is the standard induced Lebesgue volume of the subspace, denoted by $\vol_k(A)$. 
For $a \in \R^{n}$ with $\norm{a}=1$ and $t \in \R$, let $H_a^t:= \{x \in \R^{n} \mid \skalar{a}{x}=t\} = H_a + t\cdot a$ be a translated hyperplane, especially $H_a:=H_a^0$. For $\norm{a}\ne 1$ the hyperplane $H_a$ is still well defined.

We denote two special directions. 
\begin{align*}
a_{\text{min}}&:=\left(\sqrt{\frac{n}{n+1}},-\frac{1}{\sqrt{n(n+1)}},\dots,-\frac{1}{\sqrt{n(n+1)}}\right) \\
a_{\text{max}}&:=\left(\frac{1}{\sqrt{2}},0,\dots,0,-\frac{1}{\sqrt{2}}\right) 
\end{align*}
The volume of the corresponding sections can be computed elementary.
\begin{align}
\begin{split}
\vol_{n-1}(H_{a_{\text{min}}}\cap S)&=\frac{\sqrt{n+1}}{(n-1)!}\left(\frac{n}{n+1}\right)^{n-\frac{1}{2}} \gtrsim \frac{\sqrt{n+1}}{(n-1)!}\frac{1}{e} \label{eq:volmin} \\
\vol_{n-1}(H_{a_{\text{max}}}\cap S)&=\frac{\sqrt{n+1}}{(n-1)!}\frac{1}{\sqrt{2}}
\end{split}
\end{align}
The hyperplane $H_{a_{\text{min}}}$ is parallel to one of the faces of the simplex. The hyperplane $H_{a_{\text{max}}}$ contains $n-1$ vertices and the midpoint of the remaining two vertices. Due to the symmetries of the simplex the volume is invariant under permutations of the coordinates of $a$ and multiplying $a$ by $-1$.

Our results are stated in the following theorems.
The general volume formula that we find is
\begin{theorem}\label{thm:general formula}
Let $H$ be a k-dimensional subspace of $\R^{n+1}$ and $a^l$, $l=1, \dots, n+1-k$ some orthonormal basis of $H^{\perp}$.  Then
\begin{align*}
\vol_{k-1}&(H\cap S) \\ = &\frac{\sqrt{n+1-\sum_{l=1}^{n+1-k}(\sum_j a_j^l)^2}}{(k-1)!} \frac{1}{(2\pi)^{n+1-k}} \int \limits_{\R^{n+1-k}} \prod_{j=1}^{n+1} \frac{1}{1+\mathrm{i}\left( \sum_{l=1}^{n+1-k} a_j^l s^l \right)} \D s.
\end{align*}
\end{theorem}

For non-central hyperplane sections we get a result on the maximal volume.
\begin{theorem}\label{thm:maximal non central sections}
Let $0\leq K \leq 1$. For all $a \in \R^{n+1} $ with $ \norm{a}=1$ and $\sum_{j=1}^{n+1} a_j=K$ we have
\begin{equation*}
	\vol_{n-1}(H_a\cap S) \leq \frac{\sqrt{n+1-K^2}}{(n-1)!} \frac{1}{\sqrt{2-K^2}},
\end{equation*}
with equality for $a=\left( \frac{K}{2}+\sqrt{\frac{1}{2}-\frac{K^2}{4}},\frac{K}{2}-\sqrt{\frac{1}{2}-\frac{K^2}{4}},0,\dots , 0\right)$.
\end{theorem}

Concerning small central sections we prove
\begin{theorem}\label{thm:small central sections}
(i) For all $n\in \mathbb{N}$ the volume of the section $H_{a_\text{min}}$ is locally minimal, more precisely for all $a\in \R^{n+1}$ with $\norm{a}=1,\sum_{j=1}^{n+1} a_j=0$ and $a_1 \geq 0 \geq a_2,\dots,a_{n+1}$, we have
\[
\vol_{n-1}(H_a\cap S)\geq \vol_{n-1}(H_{a_{\text{min}}} \cap S).
\]
(ii) For dimensions $n=2,3$ and $4$ the volume of the section $H_{a_\text{min}}$ is globally minimal, i.e. for all $a\in \R^{n+1}$ with $\norm{a}=1,\sum_{j=1}^{n+1} a_j=0$ we have
\[
\vol_{n-1}(H_a\cap S)\geq \vol_{n-1}(H_{a_{\text{min}}} \cap S).
\]
\end{theorem}

Also for $k$-dimensional sections we find an upper bound.
\begin{theorem}\label{thm:large k dimensional sections}
Let $H$ be a $k$-dimensional subspace of $\R^{n+1}$ that contains the centroid of $S$. Then we have
\begin{equation}\label{eq:k dim bound nonopt}
\vol_{k-1} \left( H\cap S \right) \leq \frac{\sqrt{n+1}}{(k-1)!} \frac{\big(\sqrt{k}\big)^{\frac{k}{n+1}}}{\sqrt{n+1}}.
\end{equation}
If  additionally $\dist \left(H,e_j\right)^2 \leq \frac{n+1-k}{n+2-k}$ for all $j=1,\dots,n+1$, then 
\begin{equation}\label{eq:k dim bound opt}
\vol_{k-1} \left( H\cap S \right) \leq \frac{\sqrt{n+1}}{(k-1)!} \frac{1}{\sqrt{n+2-k}}
\end{equation}
holds. In this case the estimate is sharp.
\end{theorem}
We conclude the paper with a new proof of a result on irregular simplices.
\begin{theorem}\label{ex: example deformed simplex}
For odd dimension $n \geq 5$ there exist irregular simplices whose largest hyperplane section is not one of its faces.
\end{theorem}

\section{Volume formula}
\subsection{General formula}
Instead of considering the simplex $S$ we first look at \mbox{$\bar S:=\conv \{S,0\}$}, a right-angled $(n+1)$-simplex. This is a proper $(n+1)$-dimensional body in $\R^{n+1}$. The intersection $ H \cap \bar S$ is a pyramid with base $H \cap S$ and with its top in the origin. Knowing the volume of the intersection $\bar H \cap S$ we can find the volume of $H \cap S$ by the formula for the volume of a pyramid. 
We start with the height of the pyramid $H \cap \bar S$ which is given by $\dist (0, H\cap \tilde S)$, where $\tilde S:=\{x \in \R^{n+1}\mid \sum_{j=1}^{n+1}x_j=1\}$.
\begin{figure}[!h]
    \centering
        \includegraphics[width=0.35\textwidth]{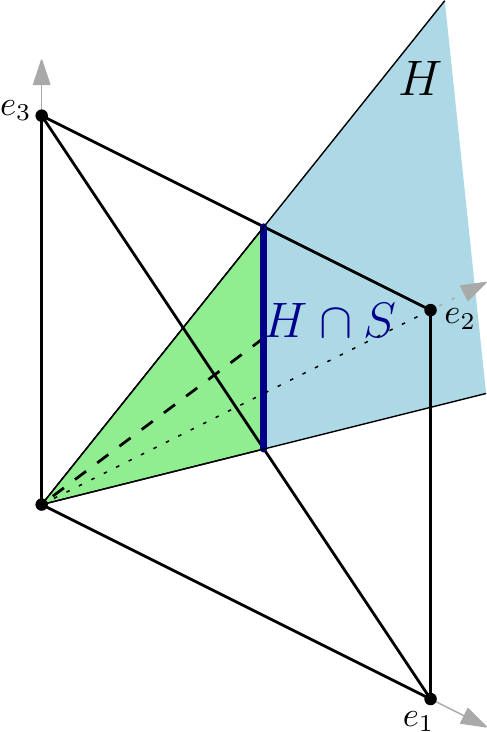}
        \caption{The embedded $2$-simplex}
				\label{fig:embedding}
\end{figure}
\begin{lemma}\label{lem:distance to 0}
Let $H$ be a k-dimensional subspace of $\R^{n+1}$ and $a^l$, $l=1, \dots, n+1-k$ some orthonormal basis of $H^{\perp}$. Then
\begin{equation*}
\dist(H\cap \tilde S,0)=\frac{1}{\sqrt{n+1-\sum_{l=1}^{n+1-k}\left(\sum_{j=1}^{n+1} a_j^l\right)^2}}.
\end{equation*}
\begin{proof}
We minimize $\norm{x}$ under the constraints $\skalar{x}{a^l}=0$ for all $l=1,\dots,n+1-k$, and $\sum_{j=1}^{n+1} x_j=1$. Define the Lagrange function
\[
\Lambda (x,\lambda,\mu):= \sum_{j=1}^{n+1} x_j^2 + \lambda \bigg(\sum_{j=1}^{n+1} x_j -1\bigg) + \sum_{l=1}^{n+1-k}\mu^l \bigg(\sum_{j=1}^{n+1} a^l_j x_j\bigg).
\]
For the derivative with respect to $x_J$ 
we find $\frac{\partial \Lambda}{\partial x_J} = 2x_J + \lambda +\sum_{l=1}^{n+1-k}\mu^l a^l_J$, 
so for a critical vector $x$ we have
\begin{equation}\label{eq:lemma distance 1}
0 = 2x_J + \lambda +\sum_{l=1}^{n+1-k}\mu^l a^l_J.
\end{equation}
Summing (\ref{eq:lemma distance 1}) over $J$ and using $\sum_{j=1}^{n+1} x_j=1$ leads to
\begin{equation}
0=2+ (n+1)\lambda  + \sum_{l=1}^{n+1-k} \mu^l \sum_{J=1}^{n+1}a_J^l.
\end{equation}
Multiplying (\ref{eq:lemma distance 1}) with $x_J$ and then summing over $J$ gives
\begin{equation}
0=2 \sum_{J=1}^{n+1} x_J^2 + \lambda.
\end{equation}
Multiplying (\ref{eq:lemma distance 1}) with $a_J^L$ for a fixed $L$ and then summing over $J$ we find
\begin{equation}
0=\lambda\sum_{J=1}^{n+1} a_J^L + \sum_{l=1}^{n+1-k}\mu^l\sum_{J=1}^{n+1}a_J^la_J^L=\lambda\sum_{J=1}^{n+1}a_J^L  + \mu^L.
\end{equation}
This implies that 
\[
\sum_{J=1}^{n+1} x_J^2 =\frac{1}{n+1-\sum_{L=1}^{n+1-k}\left(\sum_{J=1}^{n+1} a^L_J\right)^2}
\] 
is a necessary condition for an extremum. 
Obviously this critical condition gives  a minimum.
The claim follows taking the square root.
\end{proof}
\end{lemma}

We compute the volume of  $\bar H \cap S$ following an approach by M. Meyer and A. Pajor. They derived a formula for sections of the $\ell_1$-ball \cite{Meyer1988}. The $(n+1)$-simplex $\bar S$ corresponds to one orthant of the $\ell_1$-ball. The formulas look similar. Note that the $\ell_1$-ball is origin-symmetric, whereas the simplex $\bar S$ is not. Meyer's and Pajor's proof has to be modified slightly since their lemma requires symmetry of the body.
\begin{lemma}\label{lem:integral formula 1}
Let $H$ be a $k-$dimensional subspace of $\R^{n+1}$. Then
\begin{equation*}\label{eq:lem integral formula 1}
\vol_k (H\cap \bar S)= \frac{1}{k!} \int_{H\cap \R^{n+1}_{\geq 0}} \exp\bigg(-\sum_{j=1}^{n+1} x_j \bigg) \D x.
\end{equation*}
\begin{proof}
Let $H$ be defined by some orthonormal basis of $H^{\perp}$, say $a^1,\dots,a^{n+1-k}$. 
Let $K\subset H$. For $\epsilon \geq 0$ we define 
\begin{equation}\label{eq:thicken}
K(\epsilon):=\Big\{ x+\sum_{l=1}^{n+1-k} t^la^l \mid x\in K, \ t^l\in [-\epsilon,\epsilon] \Big\}.
\end{equation}
The set $K(\epsilon)$ and therefore the computations below depend on the choice of the orthonormal basis. But finally we consider limits for $\epsilon \to 0$. These are independent of the choice of $a^1,\dots, a^{n+1-k}$. 

The $k$-volume of $K$ is given by $\lim_{\epsilon\to 0}\frac{1}{(2\epsilon)^{n+1-k}}\vol_{n+1}\left(K(\epsilon)\right)$. 
For $c>0$ we have 
\begin{equation}\label{eq:good subset}
\vol_{n+1}(K(c\epsilon))=c^{n+1-k}\vol_{n+1}(K(\epsilon)).
\end{equation}

We consider the following integral with respect to the Lebesgue measure on $\R^{n+1}$: 
\begin{align*}
g(\epsilon) &:= \frac{1}{(2\epsilon)^{n+1-k}} \int_{H(\epsilon)\cap \R^{n+1}_{\geq 0}} \exp \Big(-\sum_{j=1}^{n+1}x_j \Big) \D x. 
\end{align*}
Then by integration on the level sets of $\sum_{j=1}^{n+1} x_j$ resp. by Fubini's theorem we find
\begin{align*}
						g(\epsilon)
						&=\frac{1}{(2\epsilon)^{n+1-k}} \int_{H(\epsilon)\cap \R^{n+1}_{\geq 0}} \int_{\sum_{j=1}^{n+1}x_j}^{\infty} \exp\left(-t\right) \D t \ \D x\\
						&= \frac{1}{(2\epsilon)^{n+1-k}} \int_{0}^{\infty} \vol_n \bigg(\Big\{x \in \R^{n+1}_{\geq 0} \mid x \in H(\epsilon), \sum_{j=1}^{n+1}x_j=s\Big\}\bigg) \int\limits_{s}^{\infty} \exp\left(-t\right) \D t \ \D s\\
           &= \frac{1}{(2\epsilon)^{n+1-k}} \int_{0}^{\infty}  \vol_{n+1} \bigg(\Big\{x \in \R^{n+1}_{\geq 0} \mid x \in H(\epsilon), \sum_{j=1}^{n+1}x_j\leq t \Big\} \bigg)\exp\left(-t\right) \D t .
\end{align*}
Furthermore we have
 \begin{align*}
\Big\{x \in \R^{n+1}_{\geq 0} \mid x \in H(\epsilon), \sum_{j=1}^{n+1}x_j\leq t \Big\} 
=\Big\{t\cdot x \in \R^{n+1}_{\geq 0} \mid x \in H\left(\frac{\epsilon}{t}\right), \sum_{j=1}^{n+1}x_j\leq 1 \Big\}.
	\end{align*}
	By homogeneity of $\vol_{n+1}$ we get
	\begin{align*}
&g(\epsilon) \\
&=\frac{1}{(2\epsilon)^{n+1-k}} \int_{0}^{\infty}  t^{n+1} \vol_{n+1} \bigg(\Big\{x \in \R^{n+1}_{\geq 0} \mid x \in H\left(\frac{\epsilon}{t}\right), \sum_{j=1}^{n+1}x_j\leq 1 \Big\} \bigg)\exp\left(-t\right) \D t .
\end{align*}
Since $\left\{x \in \R^{n+1}_{\geq 0} \mid x \in H\left(\frac{\epsilon}{t}\right), \sum_{j=1}^{n+1}x_j\leq 1 \right\} 
\subset
\left\{x \in \R^{n+1} \mid x \in (H\cap \bar S)\left(\frac{\epsilon}{t}\right)\right\}$ and by (\ref{eq:good subset}) we have
\begin{align*}
g(\epsilon) &\leq \frac{1}{(2\epsilon)^{n+1-k}} \int_{0}^{\infty}  t^k \vol_{n+1} \left(\left\{x \in \R^{n+1} \mid x \in (H\cap \bar S)\left(\epsilon \right)\right\} \right)\exp\left(-t\right) \D t \\
&= \frac{1}{(2\epsilon)^{n+1-k}} \vol_{n+1} \left(\left\{x \in \R^{n+1}_{\geq 0} \mid x \in (H\cap \bar S)\left(\epsilon \right)\right\} \right) \ \Gamma\left(1+k\right)
\end{align*}
For $\epsilon \to 0$ this tends to $\Gamma(1+k) \vol_k(H\cap \bar S)$.

For $\delta \geq 0$ let $M_{\delta}$ be the subset of $H\cap \bar S$ such that $M_{\delta}(\delta) \subset \bar S$.
Since the following inclusion holds
\begin{align*}
M_{\frac{\epsilon}{t}} \left(\frac{\epsilon}{t}\right)
\subset
\Big\{x \in \R^{n+1}_{\geq 0} \mid x \in H\left(\frac{\epsilon}{t}\right), \sum_{j=1}^{n+1}x_j\leq 1 \Big\}, 
\end{align*}
we get by (\ref{eq:good subset}):
\begin{align*}
g(\epsilon) &\geq \frac{1}{(2\epsilon)^{n+1-k}} \int_{0}^{\infty}  t^k \vol_{n+1} \left( M_{\frac{\epsilon}{t}} \left(\epsilon \right) \right)\exp\left(-t\right) \D t.
\end{align*}
Note that $\lim_{\epsilon \to 0} \frac{1}{(2\epsilon)^{n+1-k}} \vol_{n+1}\left( M_{\frac{\epsilon}{t}} \left(\epsilon \right)\right)=\vol_k (M_0)=\vol_k (H\cap \bar S)$. So for $\epsilon \to 0$ we have 
\[
\frac{1}{(2\epsilon)^{n+1-k}} \int_{0}^{\infty}  t^k \vol_{n+1} \left( M_{\frac{\epsilon}{t}} \left(\epsilon \right) \right)\exp\left(-t\right) \D t \ \ \longrightarrow \ \ \Gamma(1+k)\vol_k(H\cap \bar S),
\]
 since we may interchange the limit and the integral due to dominated convergence.

Therefore \[\lim_{\epsilon \to 0} g(\epsilon)= \Gamma(1+k)\vol_k(H\cap \bar S).\qedhere \]
\end{proof}
\end{lemma}

To get an explicit formula it is classical to use the Fourier transformation and the Fourier Inversion Theorem, e.g. \cite{Ball1986}, \cite{Oleszkiewicz2000}.
\begin{lemma}\label{lem:integral formula 2}
Let $H$ be a k-dimensional subspace of $\R^{n+1}$ and $a^l$, $l=1, \dots, n+1-k$ some orthonormal basis of $H^{\perp}$. Then
\begin{equation*}
 \int \limits_{H\cap \R^{n+1}_{\geq 0}} \exp \bigg(-\sum_{j=1}^{n+1} x_j\bigg) \D x = \frac{1}{(2\pi)^{n+1-k}} \int \limits_{\R^{n+1-k}} \prod_{j=1}^{n+1} \frac{1}{1 + \mathrm{i} \left( \sum_{l=1}^{n+1-k} a_j^l s^l\right)} \D s.
\end{equation*}
\begin{proof}
First rewrite the integral
\[
\int \limits_{H\cap \R^{n+1}_{\geq 0}} \exp\bigg( - \sum_{j=1}^{n+1} x_j\bigg) \D x = \int_{\forall l,j: \ \skalar{x}{a^l}=0, \ x_j\geq 0} \exp\bigg(-\sum_{j=1}^{n+1} x_j\bigg)\D x
\]
and define a map $F\colon \R^{n+1-k} \to \R$ by
\[
t=(t^1, \dots, t^{n+1-k}) \mapsto \int_{\forall l,j: \ \skalar{x}{a^l}=t^l, \ x_j\geq 0} \exp\bigg(-\sum_{j=1}^{n+1} x_j\bigg) \D x.
\]
This function is integrable, since $\int_{\R^{n+1-k}} F(t) \D t=\int_{\R^{n+1}_{\geq 0}}\exp \left(-\sum_{j=1}^{n+1} x_j\right)\D x=1$.
We apply the Fourier transform:
\begin{align*}
&
\left(2\pi\right)^{\frac{n+1-k}{2}} \hat{F}(\tau) \\
	&= 
		\int_{\R^{n+1-k}} F(t) \exp\bigg(-\mathrm{i} \skalar{t}{\tau}\bigg) \D t\\
	&= 
		\int_{\R^{n+1-k}}\left( \int_{\forall l,j: \ \skalar{x}{a^l}=t^l, \ x_j\geq 0} \exp\bigg(-\sum_{j=1}^{n+1} x_j\bigg) \ \D x\right) \exp\bigg(-\mathrm{i} \sum_{l=1}^{n+1-k} \skalar{x}{a^l}\tau ^ l\bigg) \D t\\
		&= 
	\int_{\R^{n+1}_{\geq 0}} \exp\bigg( -  \sum_{j=1}^{n+1} x_j\bigg) \exp \bigg( - \mathrm{i} \sum_ {j=1}^{n+1} \bigg(\sum_{l=1}^{n+1-k} a_j^l \tau^l\bigg) x_j \bigg) \D x\\
	&= 
	\int_{\R^{n+1}_{\geq 0}} \exp \bigg( - \sum_{j=1}^{n+1} \bigg( 1 + \mathrm{i} \sum_{l=1}^{n+1-k} a_j^l \tau^l  \bigg) x_j \bigg) \D x \\
	&= 
	\prod_{j=1}^{n+1}\frac{1}{1 + \mathrm{i} \left( \sum_{l=1}^{n+1-k} a_j^l \tau^l \right) }.
\end{align*}
This function is continuous and  integrable.
By the Fourier inversion formula we get
\begin{align*}
F(0) 
&= \frac{1}{(2\pi)^{(n+1-k)/2}} \int_{\R^{n+1-k}} \hat{F}(\tau)\D \tau  \\
&= \frac{1}{(2\pi)^{n+1-k} }\int_{\R^{n+1-k}} \prod_{j=1}^{n+1} \frac{1}{1 + \mathrm{i} \left( \sum_{l=1}^{n+1-k} a_j^l s^l\right)} \D s.\qedhere
\end{align*}
\end{proof}
\end{lemma}

Now we put together Lemmas \ref{lem:distance to 0}, \ref{lem:integral formula 1} and \ref{lem:integral formula 2} and use the formula for the $k$-volume of a pyramid
\[
\vol_k\left(H\cap \bar S\right)=\frac {1}{k} \dist(H\cap \tilde S,0) \vol_{k-1}(H\cap S)
\]
to obtain the formula from Theorem 1.

\subsection{Formula for Hyperplane sections}\label{sec:hyperplane sections}
For hyperplane sections this formula can be evaluated further by using the residue theorem. S. Webb did this for central sections and this method also works for general sections. Note that the restriction on $a$ is not an essential one, since the function $a\mapsto \vol_{n-1}(H_a\cap S)$ is continuous.
\begin{corollary}\label{cor:volume formula simplex}
Let $a  \in \R^{n+1}$ with $\norm{a}=1$ such that all $a_j>0$ are pairwise distinct and there is at least one $a_j>0$ and one $a_j <0$. Then we have
\begin{equation*}
			\vol_{n-1}(H_{a} \cap S)
	=	\frac{\sqrt{n+1-(\sum_{j=1}^{n+1} a_j)^2}}{(n-1)!} \sum_{a_j>0} \frac{1}{a_j}\prod_{k\ne j}\frac{a_j}{a_j-a_k}.
\end{equation*}
\end{corollary}

S. Webb's formula required $\sum_{j=1}^{n+1} a_j =0$. In this case the centroid of $S$ lies in the hyperplane. 
A different way of representing non-central sections is translating central sections. 
So for suitable $a \in \R^{n+1}$, $t\in \R$ with $\sum_{j=1}^{n+1} a_j = 0$  and 
$b \in \R^{n+1}$ 
we have 
\[H_a^t\cap S=H_{b}\cap S. \]
The two representations can be converted into each other:

Given $a=(a_1,\dots,a_{n+1})$ with $\norm{a}=1$, $\sum_{j=1}^{n+1} a_j=0$ and $t\in \R$ set \[b_j:=\frac{a_j-t}{\sqrt{1+(n+1)t^2}}\] for $j=1,\dots,n+1$. 
Then $b$ has norm $1$. 
Let $x\in S$. Then $\skalar{x}{a}=t$ is equivalent to $\sqrt{1+(n+1)t^2} \skalar{x}{b}= \sum_{j=1}^{n+1} (a_j x_j - tx_j) =\skalar{a}{x}-t\sum_{j=1}^{n+1} x_j=0$. So \[H_a^t \cap S = H_{b}\cap S.\]

On the other hand, given $b=(b_1,\dots,b_{n+1})$ with $\norm{b}=1$ set 
\begin{align}
\begin{split}
a_j&=\sqrt{\frac{n+1}{n+1- (\sum_{j=1}^{n+1} b_j)^2}} \cdot{} b_j - \frac{\sum_{j=1}^{n+1} b_j}{\sqrt{(n+1)(n+1-\left(\sum_{j=1}^{n+1} b_j)^2\right)}} \quad \text{ and } \\ 
t&=-\frac{\sum_{j=1}^{n+1} b_j}{\sqrt{(n+1)\left(n+1-(\sum_{j=1}^{n+1} b_j)^2\right)}}\label{eq:hyperplane convert2}
\end{split}
\end{align}
for $j=1,\dots,n+1$. Then $a$ has norm 1 and $\sum_{j=1}^{n+1} a_j=0$. For $x\in S$ we have $\skalar{x}{a}-t=\sum_{j=1}^{n+1}x_j(a_j-t)=\sqrt{\frac{n+1}{n+1- (\sum_{j=1}^{n+1} b_j)^2}} \sum_{j=1}^{n+1}x_jb_j$. Therefore the condition $ \skalar{x}{b}=0$ is equivalent to $\skalar{x}{a}=t$. So we have \[H_{b}\cap S=H_a^t \cap S.\]
Due to (\ref{eq:hyperplane convert2}) we also know
\begin{equation}\label{eq:dist c hcaps}
\dist (c, H_b\cap S)=\frac{ \left|\sum_{j=1}^{n+1} b_j\right|}{\sqrt{(n+1)\left(n+1-(\sum_{j=1}^{n+1} b_j)^2\right)}}.
\end{equation}
\section{Estimates for hyperplane sections}
We investigate the volume of a hyperplane section with fixed distance to the centroid. For a normal vector $a$  set $K:=\sum_{j=1}^{n+1} a_j$.  Due to (\ref{eq:dist c hcaps}), fixed distance means $K$ is fixed. We may assume $K \geq 0$, since $a$ and $-a$ determine the same hyperplane. By continuity of $a\mapsto \vol_{n-1}(H_a\cap S)$ we may assume that all $a_j\ne 0$. Furthermore we assume that $a=(a_1,\dots,a_P,a_{P+1},\dots,a_{n+1})$ with $a_j > 0$ for $j=1,\dots,P$ and $a_j <0$ for  $j=P+1, \dots, n+1$. This does not change the volume, since permutations of the coordinates correspond to the symmetries of the simplex.
According to Corollary \ref{cor:volume formula simplex} we have to estimate
\[
F(a):=\sum_{j=1}^P \frac{1}{a_j}\prod_{k=1,k\ne j}^{n+1} \frac{1}{1-\frac{a_k}{a_j}}
\]
as long as all positive $a_j$ are pairwise distinct. 
The estimates rely on the following inequalities, valid for $N\in\mathbb{N}$ and $x_1,\dots,x_N > 0$. These inequalities are Bernoulli's inequality resp. of the Arithmetic-Geometric-Mean inequality:
For all $N\in \mathbb{N}$ and $x_1,\dots,x_N > 0$ we have
\begin{equation}\label{eq:real inequality}
1+\sum_{j=1}^N x_j \ \leq \prod_{j=1}^N (1+x_j)  \leq \ \bigg( 1+\frac{\sum_{j=1}^N x_j}{N} \bigg)^N.
\end{equation}
The idea for our estimates is to modify a given vector $a$ such that the sum of the coordinates, their square sum and their signs do not change. Preserving the sign geometrically means the polytopal structure of the section is preserved (see subsection \ref{sec:small sections} below). Preserving the sum of the coordinates means preserving the distance of the hyperplane to the centroid.
\subsection{Maximal non-central sections}  
\begin{proposition}\label{lem:upper bound}
Let $0\leq K \leq 1$, $a \in \R^{n+1} $ with $\norm{a}=1$ and $\sum_{j=1}^{n+1} a_j=K$. Then we have
\begin{equation*}
	F(a) \leq \frac{1}{\sqrt{2-K^2}},
\end{equation*}
with equality for $a=(\frac{K}{2}+\sqrt{\frac{1}{2}-\frac{K^2}{4}},0,\dots,0,\frac{K}{2}-\sqrt{\frac{1}{2}-\frac{K^2}{4}},0,\dots,0)$.
\begin{proof}
For a given $a$ we define
\[
\tilde a:=\bigg(\gamma a_1, \dots, \gamma a_P, \beta \sum_{j> p} a_j, 0, \dots\bigg),
\]
with $\gamma \geq 0$ and $\beta \geq 0$ such that $\norm {\tilde a}=1$ and $\sum_{j=1}^{n+1}  \tilde a_j=\sum_{j=1}^{n+1} a_j$.
Then we have
\begin{align}
\gamma \sum_{j\leq P} a_j  + \beta \sum_{j>P} a_j &=\sum_{j=1}^{n+1} a_j \label{eq: rescale 1}, \\
\gamma^2 \sum_{j\leq P} a_j^2 + \beta ^2 \bigg(\sum_{j>P}a_j\bigg)^2 &=1\label{eq: rescale 2}.
\end{align}
We show that $\beta$ and $\gamma$ in $[0,1]$ with these properties exist:\\
Equation (\ref{eq: rescale 1}) describes a line $\beta_g(\gamma)$ in the $(\gamma,\beta)$-plane with positive slope, and values $\beta_g(0)=\frac{\sum_{j=1}^{n+1}a_j}{\sum_{j>P}a_j}\leq 0$ and $\beta_g(1)=1$, since $\sum_{j=1}^{n+1} a_j\geq 0$.
Equation (\ref{eq: rescale 2}) describes an ellipse. Solving this for the positive part of the ellipse, we get
$ \beta_e(0)=\sqrt{\frac {1}{(\sum_{j>P}a_j)^2}}>0$ and $\beta_e(1)=\sqrt{\frac{1-\sum_{j\leq p} a_j^2}{(\sum_{j>p}a_j)^2}}= \sqrt{\frac{\sum_{j> P} a_j^2}{(\sum_{j>p}a_j)^2}}\leq 1$. 
Therefore $\beta_g (0) \leq 0 \leq \beta_e(0)$ and $\beta_g(1)=1 \geq \beta_e(1)$. The two functions $\beta_g(\cdot)$ and $\beta_e(\cdot)$ are continuous on $\R_{\geq 0}$. Due to the intermediate value theorem they intersect, so  there are \mbox{$\beta, \gamma \in [0,1]$} with the desired properties.

We compare $F(\tilde a)$ and $F(a)$ and find
\begin{align*}
F(\tilde a) &=  \sum_{j=1}^P \frac{1}{\gamma a_j} \left( \prod_{k=1,k\ne j}^P \frac{1}{1-\frac{\gamma a_k}{\gamma a_j}}\right) \ \frac{1}{1-\frac{\beta \sum_{l >P} a_l}{\gamma a_j}}  \\		
		        & =  \sum_{j=1}^P \frac{1} {a_j} \left( \prod_{k=1,k\ne j}^P \frac{1}{1-\frac{ a_k}{ a_j}} \right)
																			\ \frac{1}{\gamma-\frac{\beta \sum_{l>P} a_l}{a_j}}  \\
						&\geq
						\sum_{j=1}^P \frac{1} {a_j} \left( \prod_{k=1,k\ne j}^P \frac{1}{1-\frac{ a_k}{ a_j}} \right)
																			\  \prod_{l>P} \frac{1}{1 - \frac{a_l}{a_j}}\\
						&= F(a).
\end{align*}
The last inequality is due to the left-hand side of (\ref{eq:real inequality}), i.e. Bernoulli's inequality, and the fact that $\beta,\gamma \in [0,1]$. More precisely, for all $j=1 \dots P$, we have
\[
\gamma-\frac{\beta \sum_{l>P} a_l}{a_j} \leq 1 + \sum_{l>P} \frac{|a_l|}{a_j} \leq \prod_{l>P} \left(1 + \frac{|a_l|}{a_j}\right)=\prod_{l>P}\left( 1 - \frac{a_l}{a_j}\right).
\]

We do the same trick for $j=P+1, \dots, n+1$. So consider a vector of the form $a=(a_1,\dots,a_P,0,\dots,0,a_{n+1})$, with $a_{j}>0$ for $j=1,\dots,P$ and $a_{n+1}<0$. We define 
\[
\tilde a := \bigg(\gamma \sum_{j=1}^{P}a_j,0,\dots,0,\beta a_{n+1} \bigg)
\]
such that $\norm{\tilde a}=1$ and $\sum_{j=1}^{n+1} \tilde a_j=\sum_{j=1}^{n+1} a_j$.  
With a similar argument as in the first step, additionally using $\sum_{j=1}^{n+1}a_j \leq 1$, we find that $\beta, \gamma \in [0,1]$ exist.

We compare $F(-a)$ and $F(-\tilde a)$. Note that the function $F$ now only has one summand. The estimates are analog to the first step.
So $F$ attains its maximum at a vector of the form $a=(a_1,0, \dots ,0, a_{P+1},0,\dots,0)$. Using $\norm{a}=1$ and $\sum_{j=1}^{n+1} a_j=K$, we conclude that
\[
a_1=\frac{K}{2}+\sqrt{\frac{1}{2}-\frac{K^2}{4}} , \quad \ \ a_{P+1}=\frac{K}{2}-\sqrt{\frac{1}{2}-\frac{K^2}{4}}
\]
and $F(a)=\frac{1}{\sqrt{2-K^2}}$.
\end{proof}
\end{proposition}
Using Proposition \ref{lem:upper bound} we immediately determine the maximal section for a fixed distance close to the centroid. Close means precisely that the distance between the hyperplane and the centroid is at most the distance from a face to the centroid. This maximal section contains $n-1$ of the vertices of the simplex.

\begin{rem}This includes the previous result of S. Webb. If we are only interested in central section, i.e. $K=0$, we do not have to introduce $\beta$ and $\gamma$ and the proof of the lemma is just by Bernoulli's inequality. Therefore we gave a different and shorter proof for Webb's result on maximal central sections.\end{rem}
\subsection{Small central sections}
A similar Lemma can be obtained for an estimate in the converse direction. Instead of concentrating the coordinates we balance them. This estimate is weaker since we can only balance the negative half of the coordinates of $a$, whereas we could concentrate the positive and the negative coordinates in the above lemma. 
\begin{proposition}\label{lem:lower bound}
Let $0\leq K$, $a \in \R^{n+1} $ with $\norm{a}=1$ and $\sum_{j=1}^{n+1} a_j=K$. Then we have
\begin{equation*}
	F(a) \geq F(\tilde a)
\end{equation*}
where $\tilde a=\big(\gamma a_1,\dots,\gamma a_P, \beta \frac{\sum_{j=P+1}^{n+1} a_j}{N}, \dots, \beta \frac{\sum_{j=P+1}^{n+1}a_j}{N}\big)$, with $\gamma,\beta$ such that $\norm{\tilde a}=1$, $\sum_{j=1}^{n+1} \tilde a_j=K$ and $N:=n+1-P$.
\begin{proof}
As in the last lemma we modify the vector $a$ to obtain an estimate for $F$. This time we choose $\beta,\gamma \geq 1$ such that for
\[
\tilde a=\bigg(\gamma a_1,\dots,\gamma a_P, \beta \frac{\sum_{j=P+1}^{n+1} a_j}{N}, \dots, \beta \frac{\sum_{j=P+1}^{N}a_j}{N}\bigg)
\]
the equations
\begin{align}
\gamma \sum_{j\leq P} a_j  + \beta \sum_{j>P} a_j &=\sum_{j=1}^{n+1} a_j \label{eq: rescale 3},\\
\gamma^2 \sum_{j\leq P} a_j^2 + \beta ^2 \frac{\left(\sum_{j>P}a_j\right)^2}{N} &=1\label{eq: rescale 4}
\end{align}
hold. 
Equation (\ref{eq: rescale 3}) is the same as (\ref{eq: rescale 1}) from the previous proof with the same implications $\beta_g(0)\leq 0$ and $\beta_g(1)=1$. Note that the slope of $\beta_g(\cdot)$ is larger than $1$.

Equation (\ref{eq: rescale 4}) again defines an ellipse. With $\beta_e (\gamma)=\sqrt{N}\sqrt{\frac{1-\gamma^2\sum_{j\leq P}a_j^2}{\left(\sum_{j>P}a_j\right)^2}}$ we find that $\beta_e(1)=\sqrt{N}\sqrt{\frac{\sum_{j>p}a_j^2}{\left(\sum_{j>P}a_j\right)^2}}\geq 1$. Furthermore $\beta_e(\gamma)=0$ for $\gamma= \frac{1}{\sqrt{\sum_{j\leq P} a_j^2}}\geq 1$.

By the intermediate value theorem, there are $\beta\geq 1$ and $\gamma \geq 1$ with (\ref{eq: rescale 3}) and (\ref{eq: rescale 4}). Since the slope of the line defined by (\ref{eq: rescale 3}) is greater than $1$, also $\beta \geq \gamma$. 

We compare $F(a)$ and $F(\tilde a)$ and obtain
\begin{align*}
F(\tilde a)&= \sum_{j=1}^P \frac {1}{\gamma a_j}\prod_{k=1,k\ne j}^P\frac {1}{1-\frac{a_k}{a_j}} \prod_{k=P+1}^{n+1}\frac {1}{1-\frac{\beta \sum_{l>P}a_l}{\gamma N a_j}}\\
&\leq \sum_{j=1}^P \frac {1}{a_j}\prod_{k=1,k\ne j}^P\frac {1}{1-\frac{a_k}{a_j}}\Bigg(\frac {1}{1-\frac{\beta \sum_{l>P}a_l}{\gamma N a_j}}\Bigg)^N.
\end{align*}
Now we use $\frac{\beta}{\gamma} \geq 1$ and the right inequality of (\ref{eq:real inequality}), which is the AGM inequality:
\[
\left(1-\frac{\beta \sum_{l>P}a_l}{\gamma N a_j}\right)^N = 
\left(1+\frac{\beta}{\gamma}\frac{\sum_{l>P}|a_l|}{N a_j}\right)^N \geq
\prod_{l>P} \left(1+\frac{|a_l|}{a_j}\right)=
\prod_{l>P}\left (1-\frac{a_l}{a_j}\right)
\]
Finally, we find that
\begin{equation*}
F(\tilde a) \leq \sum_{j=1}^P \frac {1}{a_j} \prod_{k=1, k\ne j}^P\frac {1}{1-\frac{a_k}{a_j}} \prod_{l=P+1}^{n+1}\frac {1}{1-\frac{a_l}{a_j}} = F(a). \qedhere
\end{equation*}
\end{proof}
\end{proposition}

\begin{rem}One might conjecture that the minimum of $F(a)$ without changing the signs of $a_j$ is attained in a vector of the form $\tilde a=(\xi,\dots,\xi, \eta,\dots,\eta)$, for some $\xi\geq 0, \eta < 0$. But for a vector of this form the formula from Corollary \ref{cor:volume formula simplex} and the estimates from the previous lemma do not work anymore. The example in Lemma \ref{lem:minimum small dimension} shows that this conjecture is even false for higher dimensions. \end{rem}

From Proposition \ref{lem:lower bound} we get the local result, i.e. Theorem \ref{thm:small central sections} (i).
This also immediately shows that $a_{\text{min}}$ is a global minimum for $n=2$. If we find all the local minima, in the sense of a fixed distribution of signs for the vector $a$, we can prove the global result.
For dimension $n=2,3,4$ we derive a different formula in the next section, which allows us to compute the remaining local minima. 
\subsection{Minimal central sections for small dimension}\label{sec:small sections}
We take a closer look at the geometric structure of $H_a \cap S$. For general $a$ we cannot say too much about $H_a \cap S$. The section is some polytope without obvious regularity. But if we assume $a$ to have the form $a=(a_1,\dots,a_P, \beta,\dots,\beta)$ for some $\beta<0$, as in Proposition \ref{lem:lower bound}, then we have some regularity and we get an additional volume formula. At least for small dimensions this leads to a solution of the minimal section problem.

Let $a=(a_1, \dots, a_P, a_{P+1}, \dots, a_{n+1})$ with $a_1, \dots, a_P > 0$ and $a_{P+1},\dots , a_{n+1}<0$. $P$ is the number of positive coordinates, $N:=n+1-P$ is the number of negative coordinates. The points $v_{ij}$ with $i \in \{1,\dots , P\}$ and $j\in  \{P+1,\dots, n+1\}$ are the intersection points of $H_a\cap S$ with the edges $\left[e_i , e_j\right]$ of the simplex $S$. They are of the form 
\[
v_{ij}= \Big(0,\dots,0, \frac{-a_j}{a_i-a_j},0,\dots,0,\frac{a_i}{a_i-a_j},0,\dots,0\Big),
\] 
where the $i$-th and the $j$-th coordinate are non-zero.
The section polytope $H_a\cap S$ is the convex hull of these $P\cdot N$ points $v_{ij}$.
Let us further decompose the section polytope $H_a\cap S$. It is the convex hull of polytopes of the form 
\[ K_i:=\conv \{ v_{ij}, j=P+1,\dots, n+1\} \] for $i=1,\dots,P$.

\begin{figure}[!ht]
    \centering
    \begin{minipage}[t]{0.47\linewidth}
        \centering
        \includegraphics[width=1\textwidth]{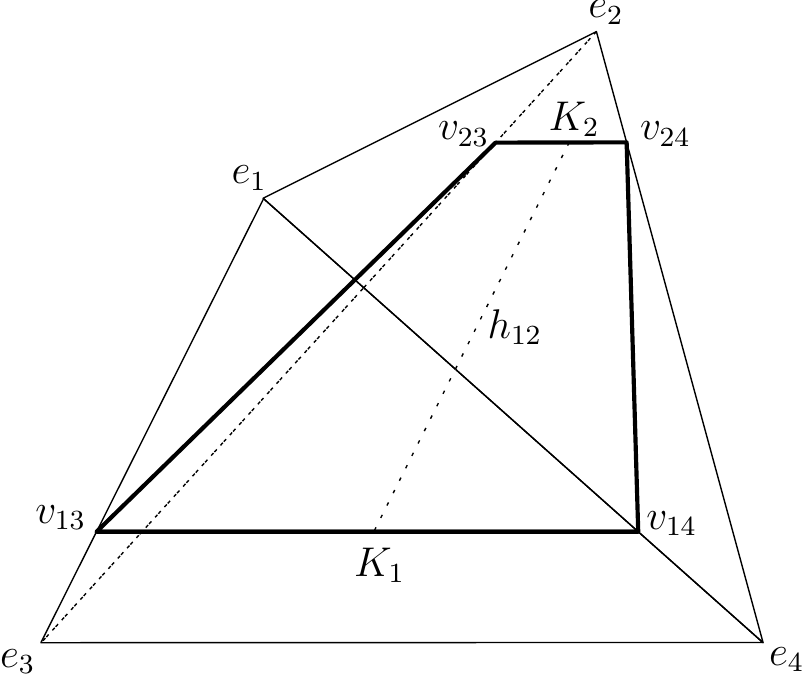}
        \caption{$H_a\cap S$ for the case $n=3$, $P=2$}
				\label{fig:balanced section 3d}
    \end{minipage}%
    \hfill
    \begin{minipage}[t]{0.47\linewidth}
        \centering
        \includegraphics[width=1  \textwidth]{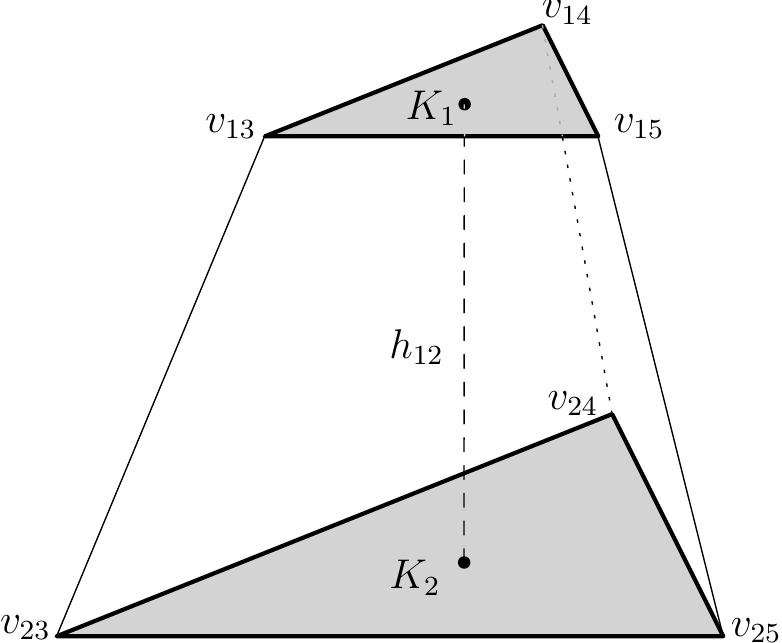}
				\caption{$H_a\cap S$ for the case $n=4$, $P=2$}
				\label{fig:balanced section 4d}
    \end{minipage}
\end{figure}

Now let $a_{P+1}=\dots=a_{n+1}=\beta$ and $\beta<0$, see Figures \ref{fig:balanced section 3d} and \ref{fig:balanced section 4d}. In this case the $K_i$ are regular simplices with $N$ vertices and side length $l_i:=\sqrt{2}\frac{a_i}{a_i-\beta}$, since for fixed $i$ and for all $j,\tilde j \in \{P+1,\dots,n+1\}$ we have
\[\norm { v_{ij}-v_{i\tilde j}}=l_i.\] 
The $(N-1)$-volume of these simplices is 
\[\vol_{N-1}(K_i)=\frac{\sqrt{N}}{(N-1)!}\left(\frac{a_i}{a_i-\beta}\right)^{N-1}.\]
Two simplices $K_i,K_{\tilde i}$ lie in parallel subspaces, 
since their edges $[v_{ij},v_{i \hat j}]$ and $[v_{\tilde i j},v_{\tilde i \hat j}]$ are parallel. The distance of $K_i$ and $K_{\tilde i}$ is given by the distance of their centroids, which is
\begin{align*}
h_{i,\tilde i} 
&:=\norm{\frac{1}{N}\sum_{j=P+1}^{n+1} v_{ij} -\frac{1}{N}\sum_{j=P+1}^{n+1}v_{\tilde i j}}\\
&=\norm{\frac{1}{N}\sum_{j=P+1}^{n+1} v_{ij} -v_{\tilde i j}} \\
&=\sqrt{ \frac{\beta^2}{(a_i- \beta)^2}+\frac{\beta^2}{(a_{\tilde i}-\beta)^2} + \frac {1}{N}\left(\frac{a_i}{a_i-\beta}-\frac{a_{\tilde i}}{a_{\tilde i}-\beta}\right)^2}.
\end{align*}
Note that $l_i$ and $h_{i \tilde i}$ are independent from the scaling of $a$. The value $\beta$ is determined by the constraints for the normal vector, up to scaling. We have $\beta=-\frac {1}{N}\sum_{j=1}^P a_j$. 
Two simplices $K_i$ and $K_{\tilde i}$ constitute a truncated pyramid, if $n=4$,  resp.  a frustum with a regular $(N-1)$-simplex as its base and another regular $(N-1)$-simplex as its top.

Let $P=2$. The section $H_a\cap S$ is the convex hull of two parallel $(N-1)$-simplices, i.e. a frustum. Let $V_1$, $V_2$ be the $(N-1)$-volume of the top resp. bottom and $h_1$, $h_2$ the relative heights. By homogeneity we know ${h_1}/{h_2}=\left({V_1}/{V_2}\right)^{1/(N-1)}$. Set $\lambda:={h_1 V_1^{-1/(N-1)}}$, then we have $\lambda={h_2V_2^{-1/(N-1)}}$.
The volume $V$ of $H_a\cap S$ is computed as 
\begin{align*}
V&=\frac {1}{N}\left(h_1 V_1-h_2V_2\right)\\
&=\frac{\lambda}{N} \left(V_1^{1+\frac {1}{N-1}}-V_2^{1+\frac {1}{N-1}}\right)\\
&=\frac{\lambda}{N}  \left(V_1^{\frac{N}{N-1}}-V_2^{\frac{N}{N-1}}\right) .
\end{align*}
Note that for $x,y >0$, $N\in\mathbb{N}$ the equation  $x^N-y^N=(x-y)\left(\sum_{m=0}^{N-1} x^{N-1-m}y^m\right)$ holds. Therefore 
\begin{align*}
V
&=\frac{\lambda \Big(V_1^{\frac {1}{N-1}}-V_2^{\frac {1}{N-1}}\Big)}{N} \sum_{m=0}^{N-1}V_1^{\frac{N-1-m}{N-1}}V_2^{\frac{m}{N-1}}\\
&=\frac{h_1-h_2}{N} \sum_{m=0}^{N-1}V_1^{\frac{N-1-m}{N-1}}V_2^{\frac{m}{N-1}}.
\end{align*}
So the volume of the section can be expressed by 
\begin{align*}
\vol_{n-1}(H_a\cap S)&=\frac{1}{N} \ h \sum_{m=0}^{n-2} \vol_{n-2}(K_1)^{\frac{n-2-m}{n-2}}\vol_{n-2}(K_2)^{\frac{m}{n-2}},
\end{align*}
with $h:=h_{1,2}$. 

Under the assumption  $P=2$, the normal vector $a$ only depends on one variable. Since the formulas are independent of scaling the vector $a$, it is sufficient to consider $a$ of the form
\begin{align*}
a=a(x):= \Big( &x, 1-x, -\frac{1}{N}, \dots,-\frac{1}{N} \Big)
\end{align*}
with $x\in \left(0,1\right)$.

So the volume is given by
\begin{align}
&V(x)\label{eq:volume formula V}:=\vol_{n-1}(H_{a(x)}\cap S)   \\	 
		&= \frac{\sqrt{N}}{N!} \sqrt{
													\frac{1}{(Nx+1)^2} 
												+ \frac{1}{(N(1-x)+1)^2} 
												+ N \left( \frac{x}{Nx+1}-\frac{1-x}{N(1-x)+1}\right)^2 
											  } \nonumber \\ \nonumber
			&\phantom{=}	\ \cdot \sum_{m=0}^{N-1} \left(\frac{Nx}{Nx+1} \right)^{{N-1-m}}
											 \left(\frac{N(1-x)}{N(1-x)+1} \right)^{{m}}.
\end{align}
For  $x=0$ and $x=1$ the geometric arguments do not work, since the involved simplices become degenerated. However, this function is still well defined. 
Note that for $x=0$ and $x=1$ the vector $a(x)$  corresponds to the vector $\left(a_{\text{min}},0\right)\in\R^{n+1}$ with $a_{\text{min}}\in \R^{n}$. 
The values of $V(0)$ and $V(1)$ equal the volume of the section for such $a(x)$, computed in (\ref{eq:volmin}). Therefore the formula may be extended to $[0,1]$.
Summarizing we proved 
\begin{lemma}
Let $a\in\R^{n+1}$, $\norm{a}=1$ and $\sum_{j=1}^{n+1} a_j=0$. Additionally let $a=\left( a_1, a_2, \beta,\dots, \beta\right)$ for some $\beta<0$ and $a_1, a_2 \geq 0$. Then with the definition (\ref{eq:volume formula V}): 
\[
\vol_{n-1}\left(H_a\cap S\right)= V\left(\frac{a_1}{a_1+a_2}\right).
\]
\end{lemma}

For normal vectors of this special form we now determine the minimal section volume.
\begin{lemma}\label{lem:minimum small dimension}
Let $V$ be as defined above. Then for $N=2,3,4$ and $x\in [0,1]$
\[V(x) \geq V\left(\frac{1}{2}\right).\]
But for $N=5$ we have $V(\frac{1}{2}) > V(0)$.
\begin{proof}
The function $V$ is differentiable. So one finds the extrema by finding the zeros of its derivative. The calculations are elementary and we just give the results.

Note that it is sufficient to consider $x\in \left[0,\frac{1}{2}\right]$ due to the symmetry of the function. For $N=2,3,4$ the derivative of $V$ has one zero in $\frac{1}{2}$ and one more zero in $ \left(0,\frac{1}{2}\right)$. The other zeros are complex or not in $\left(0,\frac{1}{2}\right)$. Comparing the function values at the critical points we know that $V$ has a global minimum in $\frac{1}{2}$.

For $N=5$ we obtain a different behavior. $V$ still has a minimum in $\frac{1}{2}$, but this is not the global minimum anymore:
\[
V(0)=\frac{125}{186624}\sqrt{5}\sqrt{42}
<\frac{625}{201684}\sqrt{10}=V\left(\frac{1}{2}\right)\qedhere
\]
\end{proof}
\end{lemma}
\begin{proof}[Proof of Theorem \ref{thm:small central sections} (ii)]
Recall that multiplying the normal vector by $-1$ or permuting the coordinates does not change the volume of the section.

If $a=(a_1,\dots,a_{n+1})$ with $a_1>0$ and $a_2,\dots,a_{n+1}<0$, by Proposition \ref{lem:lower bound} we have $\vol_{n-1}(H_a\cap S)\geq \vol_{n-1}(H_{a_{\text{min}}}\cap S)$.

For $n=2$ this is already sufficient.

Let $n=3$ and $a_1,a_2>0$ and $a_3,a_4<0$. Then by Lemma \ref{lem:minimum small dimension} with $N=2$ we know 
\[\vol_{n-1}\left(H_a\cap S\right)\geq \vol_{n-1}\left(H_{\left(\frac{1}{2},\frac{1}{2},-\frac{1}{2},-\frac{1}{2}\right)}\cap S\right)=\frac{1}{2}>\vol_{n-1}\left(H_{a_{\text{min}}}\cap S\right).\] This solves the case $n=3$.

Finally let $n=4$. For $a_1,a_2>0$ and $a_3,a_4,a_5<0$ Lemma \ref{lem:minimum small dimension} with $N=3$ yields 
\[\vol_{n-1}\left(H_a\cap S\right)\geq \vol_{n-1}\left(H_{a'}\cap S\right)=\frac{9\sqrt{6}}{125}>\vol_{n-1}\left(H_{a_{\text{min}}}\cap S\right)\]
with $a'=\sqrt{\frac{6}{5}}\left(\frac{1}{2} ,\frac{1}{2},-\frac {1}{3},-\frac {1}{3} ,-\frac {1}{3}\right)$.
\end{proof}

\begin{rem}For $\ell_p$-balls with $0<p\leq 2$ the hyperplane section volume is given by
\[
\vol_{n-1}\left(B_p^n \cap H_a\right)=C_p \int_0^{\infty} \prod_{j=1}^n \gamma_p (ta_j)\D t,
\]
where $C_p$ is some constant depending on $p$ and $n$ and $\gamma_p$ is the Fourier transform of the function $\exp (-|\cdot|^p)$ on $\R$, see \cite[Thm. 7.7]{Koldobsky2005}. Koldobsky used the fact that $\ln (\gamma_p(\sqrt{\cdot}))$ is convex on $\R_{\geq 0}$ to find the minimum and the maximum. The maximum is attained if one coordinate is $1$ and the rest is zero. The minimum is attained if all coordinates are equal.

This technique cannot be used for the simplex. The analogue to $t\mapsto \gamma_p(t)$ is the function $t\mapsto \frac{1}{1+\mathrm{i} t}$. This function is complex-valued. Furthermore the example in Lemma \ref{lem:minimum small dimension} shows that for the simplex the \emph{balanced} vector cannot be minimal in general.\end{rem}

\section{Bounds for $k$-dimensional sections}
The maximal central hyperplane section of the simplex is the one that contains $n-1$ vertices and the midpoint of the opposite edge. The natural generalization of this hyperplane section to lower dimensional sections is
\[
H\cap S=\conv \bigg\{e_1,\dots,e_{k-1},\frac{1}{n+2-k}\sum_{j=k}^{n+1}e_j\bigg\},
\] 
where $H\subset \R^{n+1}$ is a suitable $k$-dimensional subspace. 
This polytope contains $k-1$ vertices and the centroid of the opposite face. It has $(k-1)$-volume $\frac{\sqrt{n+1}}{(k-1)!} \frac{1}{\sqrt{n+2-k}}$, which is computed by the elementary volume formula for cones. 

With an extra condition we prove that this is indeed maximal. Without additional conditions we prove a bound that differs from this conjectured maximum by some factor depending on the dimension resp. codimension of the intersecting subspace. The main tool for the proof is the Brascamp-Lieb inequality in a normalized form \cite{Ball1989}.
\begin{lemma}[Brascamp-Lieb]
Let $u_1,\dots,u_{n+1}$ be unit vectors in $H\subset \R^{n+1}$, $H$ a $k$-dimensional subspace and $d_1,\dots, d_{n+1}>0$ satisfying $\sum_{j=1}^{n+1}d_j \skalar{\ \cdot \ }{u_j}u_j = \Id_{H}$. Then for integrable functions $f_1,\dots, f_{n+1}:\R \to [0,\infty)$ we have
\[
\int_{H} \prod_{j=1}^{n+1} f_j\left(\skalar{u_j}{x}\right)^{d_j} \D x\leq \prod_{j=1}^{n+1} \left(\int_{\R}f_j(t)\D t\right)^{d_j}.
\]
\end{lemma} 

\begin{proof}[Proof of Theorem \ref{thm:large k dimensional sections}]
Let $P$ be the orthogonal projection onto $H$, $d_j:=\norm{Pe_j}$ and $u_j:=\frac{Pe_j}{d_j}$.
Then $\sum_{j=1}^{n+1}d_j^2 \skalar{\ \cdot \ }{u_j}u_j=Id_H$.
Furthermore $\sum_{j=1}^{n+1} d_j^2$ is just the trace resp. the rank of the projection $P$, so $\sum_{j=1}^{n+1} d_j^2=k$. With the operator norm of the projection we observe \[\norm{Pe_j}\leq \norm{P}_{\text{op}}\norm{e_j}\leq 1.\] 
And if we require $H$ to contain the centroid $c$ we have by Pythagoras
\[\norm{Pe_j}^2=\norm{c+P(e_j-c)}^2=\norm{c}^2+\norm{P(e_j-c)}^2\geq \frac{1}{n+1}.\] 
Note that $\skalar{c}{P(e_j-c)}=\skalar{c}{e_j-c}=\skalar{c}{e_j}-\skalar c c=\frac{1}{n+1}-\frac{1}{n+1}=0$.

So we get the conditions
\begin{equation}\label{eq:constraints k-dim}
\sum_{j=1}^{n+1} d_j^2=k,  \quad  d_j\leq 1, \quad d_j \geq \frac{1}{\sqrt{n+1}} \text{ for } j=1,\dots, n+1.
\end{equation}
Now we consider the additional assumption of Theorem \ref{thm:large k dimensional sections}. Let \mbox{$\dist (H,e_j)^2 \leq \frac{n+1-k}{n+2-k}$} for all $j=1,\dots,n+1$. For each $x\in \R^{n+1}$ we have 
\[
\dist(H,x)=\norm{Px-x}=\norm{AA^t x},
\]
where $A$ is the matrix whose columns are the vectors $a^1, \dots, a^{n+1-k}$. \\
If $\norm{AA^te_j}\leq \sqrt{\frac{n+1-k}{n+2-k}}$ then $\norm{Pe_j}=\norm{1-AA^te_j}\geq \sqrt{\frac{1}{n+2-k}}$. 
Therefore
\begin{equation}\label{eq:constraints k-dim2}
 d_j \geq \frac{1}{\sqrt{n+2-k}}  \ \text{ for } j=1,\dots, n+1.
\end{equation}

Now we apply the Brascamp-Lieb inequality. For $x\in H$ we have 
\[x_j=\skalar{x}{e_j}=\skalar{Px}{e_j}=\skalar{x}{Pe_j}=\skalar{x}{d_ju_j}\] 
and therefore
\begin{align*}
		\int\limits_{H\cap \R^{n+1}_{\geq 0}} \prod_{j=1}^{n+1} \exp\left(-x_j\right) \D x 
&=	\int\limits_{H\cap \R^{n+1}_{\geq 0}} \prod_{j=1}^{n+1} \exp\left(-\skalar{x}{d_j u_j}\right) \D x \\
&=	\int\limits_{H\cap \R^{n+1}_{\geq 0}} \prod_{j=1}^{n+1} \left(\exp\left(-\frac{1}{d_j}\skalar{x}{u_j}\right)\right)^{d_j^2} \D x \\
&\leq	 \prod_{j=1}^{n+1} \left(\int_{0}^{\infty} \exp\left(-\frac{1}{d_j}s\right) \D s \right)^{d_j^2} \\
&=\prod_{j=1}^{n+1} d_j^{d_j^2}.
\end{align*}
We maximize $\prod_{j=1}^{n+1} d_j^{d_j^2}$ under the constraints given in (\ref{eq:constraints k-dim}). 
We consider the equivalent problem to maximize $F(x):=\sum_{j=1}^{n+1} x_j \ln x_j$ under the constraints 
\[
\sum_{j=1}^{n+1} x_j=k, \  \frac{1}{n+1} \leq x_j \leq 1.
\] 
The function $x\mapsto x \ln x$ is convex on $(0,1)$. Therefore also $\sum_{j=1}^{n+1} x_j \ln x_j$ is convex on $(0,1)^{n+1}$. 
The set 
\[
\bigg\{x\in [0,1]^{n+1} \mid \sum_{j=1}^{n+1} x_j=k, \ \frac{1}{n+1} \leq x_j \leq 1\bigg\}
\]
is also convex, so the maximum is attained in some extremal point of the set. The extremal points are permutations of the point
 \[x=\Big({1, \dots , 1}, {\frac{1}{n+1},\dots,\frac{1}{n+1}},\frac{k}{n+1}\Big),\]
with $k-1$ coordinates equal to $1$ and $n+1-k$ coordinates equal to $\frac 1 {n+1}$. \\
For $d_j$ this means
\[
(d_1,\dots, d_{n+1})=\Big({1, \dots, 1}, {\frac{1}{\sqrt{n+1}},\dots,\frac{1}{\sqrt{n+1}}},\sqrt{\frac{k}{n+1}}\Big),
\]
with $k-1$ coordinates equal to $1$ and $n+1-k$ coordinates equal to $\frac 1 {\sqrt{n+1}}$.\\
Therefore the maximum is $\prod_{j=1}^{n+1} d_j^ {d_j^2}=\frac{\sqrt{k}^{\frac{k}{n+1}}}{\sqrt{n+1}}$.

With the additional condition (\ref{eq:constraints k-dim2}) we find that
\[
(d_1,\dots, d_{n+1})=\Big(1, \dots, 1, {\frac{1}{\sqrt{n+2-k}},\dots,\frac{1}{\sqrt{n+2-k}}}\Big)
\]
(with $k-1$ coordinates equal to $1$ and $n+2-k$ coordinates equal to $\frac 1 {\sqrt{ n+2-k}}$)
is maximal and $\prod_{j=1}^{n+1} d_j^ {d_j^2}=\frac{1}{\sqrt{n+2-k}}$.
\end{proof}
\begin{rem}[Accuracy of the bounds]
(i) Fix $k\in\mathbb{N}$. Then the quotient of the proven bound (\ref{eq:k dim bound nonopt}) and the conjectured optimal bound (\ref{eq:k dim bound opt}) from Theorem \ref{thm:large k dimensional sections} tends to $1$ for $n\to \infty$.
So this is asymptotically  optimal.

(ii) Now we fix the codimension of the section, i.e. fix $d\in\mathbb{N}$ and let $k=n-d$. Then for $n\to \infty$ the quotient tends to $\sqrt{2+d}$.

(iii) In general, Brascamp-Lieb is a sharp inequality and its application gives sharp estimates. For example in the consideration of sections of $\ell_p$-balls one gets sharp results \cite{Barthe2001}, \cite{Ball1989}. 
The reason why we do not get a sharp bound here is the additional restriction $c\in H$.
If we do not assume the subspace $H$ to contain the centroid or even fulfill some extra condition, there is no lower bound on $d_j$ as in (\ref{eq:constraints k-dim}) or (\ref{eq:constraints k-dim2}). Then the integral is simply bounded by $1$. This bound is attained by $H=\spann(e_1,\dots,e_k)$. The corresponding section of $S$ is a $(k-1)$-face which has distance $\frac{1}{\sqrt{k}}$ from the origin. 
\end{rem}
\section{Irregular simplices}\label{chapter_irregular simplex}
For the regular simplex the maximal hyperplane section is a face, if the distance to the centroid is not prescribed. This can be proven by a theorem of Fradelizi about isotropic convex bodies. He proves that the maximal section of a cone in isotropic position is its base \cite[Corollay 3 (2.), p. 169]{Fradelizi1999}. This is still true for deformed and therefore non-isotropic simplices in dimensions $2,3$ and $4$. The case $n=2$ is obvious. The case $n=3$ was considered by \cite{Eggleston1963}. The first proof that this phenomenon does not generalize to all dimensions appears in \cite{Walkup1968}. The full solution, i.e. $n=4$ affirmative and $n\geq 5$ negative, was given by Philip \cite{Philip1972}. He uses Walkup's construction for $n=5$ and concluded by induction.

We modify our volume formula from Corollary \ref{cor:volume formula simplex} to be applicable to irregular simplices. Then we give a direct proof that for odd dimensions larger than $5$ there exist simplices such that not some of the faces is the maximal section. The constructed example uses Walkup's idea.

\subsection{Volume formula}
Every simplex is an affine image of the regular simplex. Here we denote the regular simplex by $S_{\text{reg}}$ and a general simplex by $S$. The formula follows by an application of the transformation theorem.
\begin{proposition}\label{thm: deformed simplex}
Let $S=\conv \{ v^{l} \mid l=1,\dots, n+1\}$ be an arbitrary simplex in $\R^{n+1}.$ Without loss of generality we may assume that $\sum_{j=1}^{n+1} v_j^{l}=1$ for all $l=1, \dots, n+1$, i.e. all vertices lie in the affine hyperplane defined by the regular simplex.
Let $T$ be the linear transformation that maps $S$ to the regular simplex $S_{\text{reg}}$. Let $a\in \R^{n+1}$ with $\norm{a}=1$. Then we have
\begin{align*}
\vol_{n-1} &\left(H_a\cap S\right)\\ 
= &\frac{\det \left(v^1 \dots v^{n+1}\right) }{\norm{T^{-1 \ast } a}} \ \frac{\sqrt{n+1-\left(\sum_{j=1}^{n+1} a_j\right)^2}}{\sqrt{n+1-\left( \sum_{j=1}^{n+1} \tilde a_j\right)^2}} \ \vol_{n-1} \left(H_{\tilde a} \cap S_{\text{reg}}\right),
\end{align*}
where $\tilde a := \frac{T^{-1 \ast} a}{\norm{T^{-1 \ast } a}}$.
\begin{proof}
The transformation $T$ maps $S$ to the regular simplex bijectively. Therefore $T^{-1}$ is given by 
\[
T^{-1}=(v^1 \dots v^{n+1}),
\] 
where the $v^j$ are column vectors. We analyze the behavior of the $(n-1)$-volume under $T$. 

We only sketch the proof. For $(n+1)$-dimensional subsets we have the transformation theorem in $\R^{n+1}$. 
The set $H\cap S$ is $(n-1)$-dimensional. We enlarge this set in two directions. We thicken $H\cap S$ in direction $a$, see (\ref{eq:thicken}), and we take the convex hull of this thickened set and $\{0\}$. Then we consider the image of this $(n+1)$-dimensional set and use the transformation theorem to obtain the formula from this proposition.
\end{proof}
\end{proposition}

\subsection{A simplex with a large cross section}
We show that for odd dimension $n \geq 5$ there is a simplex whose largest hyperplane section is not one of its faces.
\begin{proof}[Proof of Theorem \ref{ex: example deformed simplex}]
We consider the regular simplex with $n+1$ vertices, where $n+1$ is even. Then we take a hyperplane such that on both sides of the hyperplane the vertices build a $\frac{n+1}{2}$-simplex parallel to the hyperplane. Now we compress the simplex along the normal vector of this hyperplane. The intersection with the simplex remains the same but the faces of the compressed simplex become smaller. With the formula from Proposition \ref{thm: deformed simplex}  we compute the volumes and show that for $n\geq 5$ the compressed simplex has the desired property for a certain degree of compression.

We make this construction explicit. For $-\frac{1}{n+1} <\delta \leq 0$ let
\[
	T^{-1}:=\left(\begin{array} {cccc|cccc}
			1 +\delta 	& \delta 				& \dots  	& \delta 	& -\delta  	& \dots &\dots & -\delta\\
			   \delta 	& 1 + \delta 		&  				& \vdots &   \vdots  &  &   &    \vdots  \\
					\vdots			&					&		\ddots		&	\delta			&				\vdots		  	& 	& &	 \vdots		\\
				 \delta &\dots &\delta &1+\delta & 								-\delta	&\dots	&\dots &	-\delta\\ \hline
			-\delta & \dots&\dots &-\delta		& 1+\delta	&\delta &\dots &\delta\\
			\vdots & &  &\vdots					&	\delta 	&\ddots &   &\vdots\\
			\vdots  &   &  & \vdots 		& \vdots & &\ddots &\delta\\
			-\delta &\dots  &\dots & -\delta &\delta &\dots &\delta &1+\delta
								
	\end{array}	\right)
\]
The columns the matrix $T^{-1}$ define the vertices of a deformed simplex $S(\delta)$.
Note that $S(0)=S_{\text{reg}}$ and $S\left(-\frac{1}{n+1}\right)$ is  a degenerate simplex which is a $(n-1)$-dimensional set with $n$-dimensional volume equal to $0$.

With $v^{\ast}=\sqrt{\delta}\left(1, \dots,1,-1, \dots,-1\right)$, we may also write $T^{-1}=(\Id + vv^{\ast})$. Using basic matrix theory we know that $T=\Id - \frac{1}{1+v^{\ast}v}v v^{\ast}$ and $\det (T^{-1})= \det{\Id} (1+v^{\ast}v)$. So 
\[
\det(T^{-1})=1 +(n+1)\delta.
\]
Let 
\begin{align*}
a&:=\frac{1}{\sqrt{n+1}}(1,\dots,1,-1,\dots, -1),\\
b&:=\frac{1}{\sqrt{(1+n\delta)^2+n\delta^2}} ( 1+n\delta,-\delta,\dots,-\delta,\delta,\dots,\delta).
\end{align*}
We compute the vectors $\tilde a$ and $\tilde b$ according to Proposition \ref{thm: deformed simplex}. We find
 $\tilde a=a$ and $\tilde b=(1,0,\dots,0)$.
The intersections $H_{\tilde a}\cap S_{\text{reg}}$ resp. $H_a\cap S(\delta)$ are the central sections described at the beginning of the proof. The intersection $H_{\tilde b}\cap S_{\text{reg}}$ is a face of the regular simplex. Therefore $H_b\cap S(\delta)$ is a face of the simplex $S(\delta)$, since $T$ maps faces to faces. 
All faces of the simplex are of this form and therefore have the same volume.

It remains to compute the ratio of the section volumes using the formulas from Theorem \ref{thm:general formula} and analyze the behavior for $\delta \to -\frac{1}{n+1}$. We find
\begin{align*}
\frac{\vol_{n-1}(H_a\cap S(\delta))}{\vol_{n-1}(H_b\cap S(\delta))} \rightarrow \frac{(n+1)(n-1)!}{2((n-1)!!)^2}.
\end{align*}
For odd  $n\geq 5$ this quotient is larger than $1$. 
So for some $\delta \in \left(-\frac{1}{n+1},0\right)$ the simplex $S(\delta)$ has the desired property.
\end{proof}
\section*{Acknowledgements} This work is part of my PhD Thesis. I thank my advisor Hermann K{\"o}nig for his support and advice. 
My research was partly supported by DFG (project \mbox{KO 962/10-1}).

\end{document}